\documentclass[fleq]{amsart}
\usepackage{amssymb,amsmath,latexsym,amsbsy,color,enumerate}
\numberwithin{equation}{section}

\newtheorem{theorem}{Theorem}

\newtheorem{lemma}{Lemma}
\newtheorem{remark}{Remark}
\newtheorem{proposition}{Proposition}
\newtheorem{corollary}{Corollary}

\begin{document}
 \title{Degree complexity of birational maps related to matrix inversion: Symmetric case}
\author{Tuyen Trung Truong}
    \address{Indiana University Bloomington IN 47405}
 \email{truongt@indiana.edu}
\thanks{}
    \date{\today}
    \keywords{Birational Mappings; Degree Complexity; Matrix Inversion; Symmetric Matrices.}
    \subjclass[2000]{37F99, 32H50.}
    \begin{abstract}
For $q\geq 3$, we let $\mathcal{S}_q$ denote the projectivization of the set of symmetric $q\times q$ matrices with coefficients in $\mathbb{C}$. We let
$I(x)=(x_{i,j})^{-1}$ denote the matrix inverse, and we let $J(x)=(x_{i,j}^{-1})$ be the matrix whose entries are the reciprocals of the entries of $x$.
We let $K|\mathcal{S}_q=I\circ J:~\mathcal{S}_q\rightarrow \mathcal{S}_q$ denote the restriction of the composition $I\circ J$ to $\mathcal{S}_q$. This
is a birational map whose properties have attracted some attention in statistical mechanics. In this paper we compute the degree complexity of
$K|\mathcal{S}_q$, thus confirming a conjecture of Angles d'Auriac, Maillard, and Viallet in [J. Phys. A: Math. Gen. 39 (2006), 3641--3654].
\end{abstract}
\maketitle
\section{Introduction}       
\label{SecIntroduction}

Fix $q\geq 3$, let $\mathcal{M}_q$ denote the space of $q\times q$ matrices with coefficients in $\mathbb{C}$, and let $\mathbb{P}(\mathcal{M}_q)$ denote
its projectivization. Then the mapping $K:\mathbb{P}(\mathcal{M}_q)\rightarrow \mathbb{P}(\mathcal{M}_q)$ is defined as follows: $K=I\circ J$, where
$J(x)=(x_{i,j}^{-1})$ takes the reciprocal of each entry of the matrix $x=(x_{i,j})$, and $I(x)=(x_{i,j})^{-1}$ is the matrix inverse. The map $K$ is of
interest since it represents a basic symmetry in certain problems of lattice statistical mechanics, and has been studied in
\cite{auriac-maillard-viallet1}, \cite{auriac-maillard-viallet2}, \cite{bedford-kim1}, \cite{bedford-kim2}, \cite{bedford-tuyen},
\cite{bellon-maillard-viallet}, \cite{bellon-viallet}, \cite{boukraa-maillard}, and \cite{preissmann-auriac-maillard}.

The degree complexity of $K$ is the exponential rate of growth of the degrees of its iterates:
\begin{equation}
\delta (K)=\lim _{n\rightarrow\infty}(deg(K^n))^{1/n}. \label{EquationFirstDynamicalDegree}\end{equation}

There are many $K$-invariant subspaces $\mathcal{T}\subset \mathbb{P}(\mathcal{M}_q)$. The first were considered are $\mathcal{S}_q$ (the space of
symmetric matrices), $\mathcal{C}_q$ the cyclic (also called circulant) matrices, and $\mathcal{SC}_q=\mathcal{S}_q\cap \mathcal{C}_q$ (see
\cite{preissmann-auriac-maillard} for more $K$-invariant subspaces of $\mathbb{P}(\mathcal{M}_q)$). In view of complex dynamics, as well as physical
meaning, the map $K$ as well as the restrictions of $K$ to invariant spaces are of interest. One of the basic questions is to determine the degree
complexities $\delta (K|\mathcal{T})$. The values $\delta (K|\mathcal{C}_q)$ were found in \cite{bellon-viallet} and \cite{bedford-kim2}; the values of
$\delta (K|{\mathcal{SC}_q})$ were found in \cite{auriac-maillard-viallet2} for prime $q$'s, and in \cite{bedford-kim2} for general $q$'s. Based on
extensive computations, \cite{auriac-maillard-viallet2} has conjectured that
\begin{equation}
\delta (K|\mathcal{C}_q)=\delta (K)=\delta (K|\mathcal{S}_q) ,\label{EquationConjecture}\end{equation} for all $q$.

In \cite{bedford-tuyen}, we proved that $\delta (K)=\delta (K|\mathcal{C}_q)$. In this paper we prove the remaining conjectured equality.
\begin{theorem}
$\delta (K|\mathcal{S}_q)=\delta (K)=\delta (K|\mathcal{C}_q)$ is the largest modulus of the roots of the polynomial $\lambda ^2-(q^2-4q+2)\lambda
+1$.\label{theorem1}\end{theorem}

The proof of Theorem \ref{theorem1} is similar to the proofs for other cases (general matrices, $\mathcal{C}_q$, $\mathcal{SC}_q$) in that we repeatedly
blowup subvarieties to construct a space $Z\rightarrow \mathbb{P}(\mathcal{S}_q)$, and we conclude by showing that $\delta (K)$ equals the spectral
radius $sp(K_Z^*)$ of the pullback operator $K_Z^*:Pic(Z)\rightarrow Pic(Z)$ for the lifted map $K_Z:Z\rightarrow Z$. However, the behavior of singular
orbits is much more complicated for the symmetric case that we consider here. Let us give a brief comparison of these proofs in the following.

The computations of $\delta (K|\mathcal{C}_q)$ and $\delta (K|{\mathcal{SC}_q})$ can be reduced to computations of $\delta (F)$ where $F=L\circ J$ for
appropriate linear maps $L$. It was shown in \cite{bedford-kim1} (respectively \cite{bedford-kim2}) that after a finite series of blowups $Z\rightarrow
\mathcal{C}_q$ (respectively $Z\rightarrow \mathcal{SC}_q$), the induced maps $F_Z$ on $Z$ is algebraic stable, i.e. satisfy
\begin{equation}
(F_Z^n)^*=(F_Z^*)^n, \label{Equation11Regular}\end{equation} for all $n\in \mathbb{N}$, as linear maps on $Pic (Z)$. It follows (see for example
\cite{fornaess-sibony}) that $\delta (F)$ is the spectral radius $sp(F_Z^*)$ of $F_Z^*$.

For the case of general matrices, we constructed in \cite{bedford-tuyen} a space $Z$ for which $sp(K_Z^*)=\delta (K|\mathcal{C}_q)$. This immediately
implies $\delta (K)=sp(K_Z^*)=\delta (K|\mathcal{C}_q)$.  (Remark: The same argument as that of the proof of Lemma
\ref{LemmaOrbitsOfExceptionalHypersurfacesOfKZ} below shows that in fact the map $K_Z$ in \cite{bedford-tuyen} satisfies condition
(\ref{Equation11Regular}), thus gives another proof to the cited result in \cite{bedford-tuyen}.)

For the proof of Theorem \ref{theorem1} in this paper, we will construct a space $Z$ via a construction which is similar to, but more complicated than,
the one in \cite{bedford-tuyen}. Although we do not prove (\ref{Equation11Regular}), we show that $\delta (K|\mathcal{S}_q)=\delta (K)=\delta
(K|\mathcal{C}_q)$ are all equal to the spectral radius of $K_Z^*$. The results that allow us to circumvent (\ref{Equation11Regular}) in this case are
Proposition \ref{Theorem11RegularOnOpenSet} and Theorem \ref{TheoremLowerBoundForLambda1}.

This paper is organized as follows: In Section 2, we give some basic properties of the map $K|\mathcal{S}_q$. In Section 3 we construct a space $Z$ by a
series of blowups starting from $\mathcal{S}_q$. In Section 4 we explore the behavior of the iterates of the map $K_Z$ on the exceptional hypersurfaces,
and obtain a lower bound for $\delta (K|\mathcal{S}_q)$. In Section 5 we show that the lower bound is equal to the largest modulus of the roots of the
polynomial $\lambda ^2-(q^2-4q+2)\lambda +1$, thus complete the proof of Theorem \ref{theorem1}.

{\bf Acknowledgement.} The author would like to thank Professor Eric Bedford for introducing the topic of this paper, and for his constant help and
encouragement in the course of this project. The author also would like to thank the referee for many helpful comments that helped to improve the paper.

\section{Basic properties of the map ${K}$}
By \cite{bedford-tuyen}, we know that $1\leq \delta (K|\mathcal{S}_q)\leq \delta (K)\leq 1$ for $q=2,3,4$, so in the sequel we will assume that $q\geq
5$. For convenience we will use the simple notation ${K}$ for $K|\mathcal{S}_q$.

First, we introduce some notation that will be helpful in the course of the proof of Theorem \ref{theorem1}. Most of the notation used here have a
counterpart for the case of general matrices, which was used in \cite{bedford-tuyen}.

For $1 \leq j\leq q-1$, define $R_j$ to be the set of matrices in $\mathcal{S}_q$ of rank less than or equal to $j$. Elements of $R_1$, the symmetric
matrices of rank $1$, may be represented as $\nu \otimes \nu =(\nu _i\nu _j)_{1\leq i,j\leq q}$ for $\nu =(\nu _1,\ldots ,\nu _q)\in \mathbb{C}^q$. In
particular, $R_1$ is a smooth subvariety of $\mathcal{S}_q$.

For $i,j=1,\ldots ,q$ denote:
\begin{eqnarray*}
\Sigma _{i,j}=\{x=(x_{k,l})\in S_q:~x_{i,j}=0\},
\end{eqnarray*}
and define
\begin{eqnarray*}
A _{i,j}=\bigcap _{k=i\mbox{ or }l=j}\Sigma _{k,l}.
\end{eqnarray*}
Thus $\Sigma _{i,j}$ is the set of symmetric matrices whose $(i,j)$-th entry is zero, and $A_{i,j}$ is the set of symmetric matrices whose $i$-th and
$j$-th rows and columns are zero. In particular, $A_{i,j}=A_{i,i}\cap A_{j,j}$ for all $1\leq i,j\leq q$. This leads to a difficulty that does not arise
in the non-symmetric case.

We summarize some properties of the map ${K}$ in the following proposition
\begin{proposition}
a) The exceptional hypersurfaces of ${K}$ are $JR_{q-1}$ and $\Sigma _{i,j}$'s.

b) The indeterminacy locus ${K}$ is contained in the set
$$JR_{q-2}\cup \bigcup _{(i,j)\not= (k,l)}(\Sigma _{i,j}\cap \Sigma _{k,l}).$$

c) $deg({K})=q^2-q+1$.
 \label{PropositionBasicPropertiesOfK}\end{proposition}
\begin{proof}
The proofs of a) and b) are similar to those of Propositions 2.1 and 3.1 in \cite{bedford-tuyen} (see also the results in Section 3 of this paper).

We now proceed to proving c). Regarding $\mathcal{S}_q$ as the projective space $\mathbb{P}^{(q^2+q-2)/2}$, then a point $y\in \mathcal{S}_q$ can be
represented by the homogeneous coordinates  $(y_{i,j},~ 1\leq i\leq j\leq q)$. Then the corresponding matrix in $\mathcal{M}_q$ is the symmetric matrix
$\widehat{y}$ whose entries are $\widehat{y}_{i,j}=y_{i,j}$ for $1\leq i\leq j\leq q$.

It suffices to show that the homogeneous representation $\widehat{K}$ of $K$ is:
\begin{eqnarray*}
\widehat{K}_{i,j}(y)=C_{i,j}(1/\widehat{y})\prod (\widehat{y}),
\end{eqnarray*}
for $1\leq i\leq j\leq q$, where $\prod (\widehat{y}):=\prod _{1\leq i,j\leq q}\widehat{y}_{k,l}$ and $C_{i,j}(1/\widehat{y})$ is the $(i,j)$-cofactor of
the matrix $1/\widehat{y}$. That is, to show that the GCD of all polynomials $\widehat{K}_{i,j}(y)$ (for $1\leq i\leq j\leq q$) is $1$. To this end, it
suffices to show that the GCD of all polynomials $\widehat{K}_{i,i}(y)$ (where $1\leq i\leq q$) is $1$.

Note that the rational function $C_{i,i}(1/\widehat{y})$ does not depend on the variables $\widehat{y}_{i,k}$ and $\widehat{y}_{k,i}$ for $1\leq k\leq
q$. Moreover, since $C_{i,i}(1/\widehat{y})$ is the determinant of the $(q-1)\times (q-1)$ symmetric matrix obtained by deleting the $i$-th row and
$i$-th column from the matrix $1/\widehat{y}$, it is easy to see that
$$D_i(y):=C_{i,i}(1/\widehat{y})\prod _{(k-i)(l-i)\not= 0}\widehat{y}_{k,l}$$
is a polynomial independent of variables $\widehat{y}_{i,k}$ and $\widehat{y}_{k,i}$ for $1\leq k\leq q$, and is not divisible by any of the variables
$\widehat{y}_{k,l}$ where $1\leq k,l\leq q$. Then we have
\begin{eqnarray*}
\widehat{K}_{i,i}(y)=D_{i}(y)E_i(y)
\end{eqnarray*}
where $E_i(y)=\prod _{(k-i)(l-i)=0}\widehat{y}_{k,l}$. Observe that

1). For any $i$ and $j$, $GCD(D_i,E_j)=1$. This is because as noted above, $D_i$ is not divisible by any of the variables $\widehat{y}_{k,l}$, while
$E_j$ is a monomial in these variables.

2). $GCD(E_1,E_2,\ldots ,E_q)=1$. In fact, $E_i$ depends only on the variables in $S_i=\{\widehat{y}_{i,1},\widehat{y}_{i,2},\ldots
,\widehat{y}_{i,q}\}$. Hence if $\phi$ is a divisor of $E_i$, $\phi$ depends only on the variables in $S_i$. Since $\bigcap _{i=1,\ldots
,q}S_i=\emptyset$, it follows that the $GCD(E_1,\ldots ,E_q)$ must be a constant.

3). $GCD(D_1,\ldots ,D_q)=1$. The argument is similar to that of 2).

From 1), 2) and 3), it follows that $GCD(\widehat{K}_{1,1},\widehat{K}_{2,2},\ldots ,\widehat{K}_{q,q})=1$.
\end{proof}
\section{Construction of the space $Z$}

Let us describe the sequence of blowups used to construct $Z$.

A) First we let $\pi _1:Z_1\rightarrow \mathcal{S}_q$ be the blowing up with center $R_1$ and exceptional divisor $\mathcal{R}^1=\pi _1^{-1}(R_1)$. To
give a local coordinate system we fix $2\leq i_0,j_0\leq q$, $1\leq k_0\leq q$. Let $s\in \mathbb{C}$; $v=(v_{i,j})_{2\leq i,j\leq q}\in
\mathcal{S}_{q-1}$ and $v_{i_0,j_0}=1$; $\nu =(\nu _1,\ldots ,\nu _q)\in \mathbb{C}^q$ and $\nu _{k_0}=1$, and $\nu \otimes \nu \in \mathcal{M}_q$ whose
$(i,j)$-th entry is $\nu _i\nu _j$. Without loss of generality, we may assume that $k_0=1$, i.e. $\nu _1=1$. Then, in the local coordinate $(s,v, \nu )$
the projection $\pi _1=\pi _{\mathcal{R}^1}$ is given by
\begin{equation}\pi _{\mathcal{R}^{1}}(s,v,\nu )=\nu
\otimes \nu +s\left (\begin{array}{ll}0 &0 \\0 &v\end{array}\right ).\label{EquationLocalCoordinateForR1}\end{equation} In this local coordinate system,
$\mathcal{R}^{1}=\{s=0\}$.

B) Next we let $\pi _2:Z_2\rightarrow Z_1$ be the blow up of $Z_1$ along the strict transforms of $A_{i,j}$ for all $1\leq i<j\leq q$. The space $Z_2$
depends on the order in which these blowups are performed. But it does not matter for our purpose, the Picard group $Pic(Z_2)$ of $Z_2$ is generated by
$Pic(Z_1)$ and the exceptional divisors $\mathcal{A}^{i,j}=\pi _2^{-1}(A_{i,j})$. The object we will use is $Pic(Z_2)$, which is essentially independent
of the order of blowups. We describe a local coordinate system of $\pi _2$ near the exceptional divisor $\mathcal{A}^{1,2}$. We fix $3\leq i_0,j_0\leq
q$, $1\leq \min \{k_0,l_0\}\leq 2$. Let $s\in \mathbb{C}$; $v=(v_{i,j})_{3\leq i,j\leq q}\in \mathcal{S}_{q-2}$ and $v_{i_0,j_0}=1$;
$$\left (\begin{array}{llll}\zeta _{1,1}&\zeta _{1,2}&\ldots&\zeta _{1,q}\\ \zeta _{2,1}&\zeta _{2,2}&\ldots&\zeta _{2,q}\\ \vdots &\vdots&0_{q-2}& \\ \zeta _{q,1}&\zeta _{q,2}&&\end{array}\right
)=:\left (\begin{array}{lll}\zeta &\zeta &\zeta\\ \zeta &\zeta &\zeta  \\ \zeta &\zeta &0_{q-2}\end{array}\right )\in \mathcal{S}_q,$$ where $0_{q-2}$ is
the $(q-2)\times (q-2)$ zero matrix; $\zeta =(\zeta _{k,l})_{1\leq \min\{k,l\}\leq 2}$, and $\zeta _{k_0,l_0}=1$. In the local coordinate $(s,\zeta , v
)$, the projection $\pi _2=\pi _{\mathcal{A}^{1,2}}$ is given by
\begin{equation}
\pi _{\mathcal{A}^{1,2}}(s,\zeta ,v)=\left (\begin{array}{lll}s\zeta&s\zeta&s\zeta \\s\zeta&s\zeta&s\zeta \\s\zeta&s\zeta&v\end{array}\right
).\label{EquationLocalCoordinateForA12}
\end{equation}
In this local coordinate system, $\mathcal{A}^{1,2}=\{s=0\}$. Local coordinates near other $\mathcal{A}^{i,j}$'s ($i\not= j$) are similarly defined.

C) Next we let $\pi _3:Z_3\rightarrow Z_2$ be the blow up of $Z_2$ along the strict transforms of $A_{i,i}$ for all $1\leq i\leq q$, with exceptional
divisors $\mathcal{A}^{i,i}=\pi _3^{-1}(A_{i,i})$. We describe a local coordinate system of $\pi _2$ near the exceptional divisor $\mathcal{A}^{1,1}$. We
fix $2\leq i_0,j_0\leq q$, $1\leq k_0\leq q$. Let $s\in \mathbb{C}$; $v=(v_{i,j})_{2\leq i,j\leq q}\in \mathcal{S}_{q-1}$ and $v_{i_0,j_0}=1$;
 $\zeta =(\zeta _{k,l})_{\min \{k,l\}=1}$ and $\zeta _{1,k_0}=1$. In the
local coordinate $(s,\zeta , v )$, the projection $\pi _3=\pi _{\mathcal{A}^{1,1}}$ is given by
\begin{equation}
\pi _{\mathcal{A}^{1,1}}(s,\zeta ,v)=\left (\begin{array}{ll}s\zeta &s\zeta \\s\zeta &v\end{array}\right ).\label{EquationLocalCoordinateForA11}
\end{equation}
In this local coordinate system, $\mathcal{A}^{1,1}=\{s=0\}$.

Let $K_{Z_3}=\pi _{Z_3}^{-1}\circ K\circ \pi _{Z_3}$ be the induced map of $K$ in $Z_3$.

\begin{proposition}

i) ${K}_{Z_3}(\mathcal{R}^1)=R_{q-1}$.

ii) ${K}_{Z_3}(JR_{q-1})=\mathcal{R}^1$.

iii) For all $1\leq i\leq q$, ${K}_{Z_3}(\Sigma _{i,i})=\mathcal{A}^{i,i}$.

iv) For all $1\leq i< j\leq q$, ${K}_{Z_3}(\Sigma _{i,j})=\mathcal{A}^{i,j}\cap \Sigma _{i,i}\cap \Sigma _{j,j}$.
\label{PropsitionBasicPropertyOfKZ3}\end{proposition}
\begin{proof}

i) It suffices to show that: for $\nu =(1,\nu _2,\ldots ,\nu _q)$, $z=\pi _{\mathcal{R}^1}(0,v,\nu )\in \mathcal{R}^1$ then

$$K_{Z_3}(z)=A^t\left (\begin{array}{ll}0 &0 \\0 &I_{q-1}(v')\end{array}\right )A,$$
where $I_{q-1}$ is the matrix inverse on $\mathcal{M}_{q-1}$,
$$v'=\left (-\frac{v_{j,k}}{\nu _j^2\nu _k^2}\right )_{2\leq j,k\leq q},
~A=\left (\begin{array}{llll}1 &0&\ldots &0 \\-\frac{1}{\nu _2} &1&&\\ \vdots & &\ddots &\\-\frac{1}{\nu _q}&&&1\end{array}\right ),$$ and $A^t$ is the
transpose of $A$. Here the entries of $A$ outside the main diagonal and the first column are zero.

Without loss of generality, we work at $v$ and $\nu$ such that $v'$ in the above is invertible. We have
$$J(\pi _{\mathcal{R}^1}(s,v,\nu ))=\frac{1}{\nu \otimes \nu}+sv'+O(s^2)=\pi _{\mathcal{R}^1}(s+O(s^2),v'+O(s),\frac{1}{\nu}).$$
Let $e_{1}=(1,0,\ldots ,0)$ be the first standard basis vector in $\mathbb{C}^q$. Then
$$A\left (\frac{1}{\nu \otimes \nu}\right )A^t=A\left (\frac{1}{\nu}\otimes \frac{1}{\nu}\right )A^t=
\left (A \frac{1}{\nu}\right )\otimes \left (A \frac{1}{\nu}\right )=e_1\otimes e_1=\left (\begin{array}{ll}1 &0\\0&0\end{array}\right ).$$

Since $A_{[1,1]}$ (respectively $A^t_{[1,1]}$), the matrix in $\mathcal{M}_{q-1}$ obtained by deleting the first row and column of $A$ (correspondingly
of $A^t$), is the identity matrix in $\mathcal{M}_{q-1}$, we obtain:

$$sAv'A^t=\left (\begin{array}{ll}0 &0\\0&sA_{[1,1]}v'A^t_{[1,1]}\end{array}\right )=\left (\begin{array}{ll}0 &0\\0&sv'\end{array}\right ).$$

Hence
\begin{eqnarray*}
K_{Z_3}(z)&=&\pi _{Z_3}^{-1}\circ I\circ J\circ \pi _{Z_3}(z)\\
&=&\pi _{Z_3}^{-1}\circ I(\frac{1}{\nu \otimes \nu}+sv'+O(s^2))\\
&=&\pi _{Z_3}^{-1}(A^tI[A(\frac{1}{\nu \otimes \nu}+sv'+O(s^2))A^t]A).
\end{eqnarray*}

The principal part (first terms of Taylor expansion) of the latter is equal to
$$\pi _{Z_3}^{-1}(A^tI\left (\begin{array}{ll}1 &0\\0&sv'\end{array}\right )A)
=\pi _{Z_3}^{-1}(A^t\left (\begin{array}{ll}s &0\\0&I_{q-1}(v')\end{array}\right )A),$$ and i) follows by letting $s\rightarrow 0$.

Proofs of ii), iii), and iv) are similar (cf. \cite{bedford-tuyen}, Sections 2 and 3).
\end{proof}

\begin{remark}
Proposition \ref{PropsitionBasicPropertyOfKZ3} iv) shows that $\Sigma _{i,j}$ ($i<j$) is still exceptional for the map $K_{Z_3}$, which differs from the
corresponding situation in \cite{bedford-tuyen} for general matrices. This motivates us to perform blowups in subsection E) below.
\end{remark}

D) Next we let $\pi _4:Z_4\rightarrow Z_3$ be the blow up of $Z_3$ along the strict transforms of $B_{i,i}=\mathcal{A}^{i,i}\cap \Sigma _{i,i}$ (where
$1\leq i\leq q$), with exceptional divisors $\mathcal{B}^{i,i}=\pi _4^{-1}(B_{i,i})$. We describe two local coordinate systems of $\pi _4$ near the
exceptional divisor $\mathcal{B}^{1,1}$.

For the first local coordinate system, we fix $2\leq i_0,j_0\leq q$, $1\leq k_0\leq q$. Let $t, \xi\in \mathbb{C}$; $v=(v_{i,j})_{2\leq i,j\leq q}\in
\mathcal{S}_{q-1}$ and $v_{i_0,j_0}=1$;
 $\zeta =(\zeta _{k,l})_{\min \{k,l\}=1,~k\not=
l}$ and $\zeta _{1,k_0}=1$. In the local coordinate $(t,\xi ,\zeta , v )$, the projection $\pi _4=\pi ^1_{\mathcal{B}^{1,1}}$ is given by
\begin{equation}
\pi ^1_{\mathcal{B}^{1,1}}(t,\xi ,\zeta ,v)=\left (\begin{array}{ll}t^2\xi&t\zeta \\t\zeta&v\end{array}\right ).\label{EquationLocalCoordinateForB11}
\end{equation}
In this local coordinate system, $\mathcal{B}^{1,1}=\{t=0\}$.

To cover the points corresponding to $\xi =\infty$ in the first projection $\pi ^1_{\mathcal{B}^{1,1}}$, we let $t,\xi \in \mathbb{C}$;
$v=(v_{i,j})_{2\leq i,j\leq q}\in \mathcal{S}_{q-1}$ and $v_{i_0,j_0}=1$;
 $\zeta =(\zeta _{k,l})_{\min \{k,l\}=1,~k\not= l}$ and $\zeta _{1,k_0}=1$.
In the local coordinate $(t,\xi ,\zeta , v )$, the projection $\pi _4=\pi ^2_{\mathcal{B}^{1,1}}$ is given by
\begin{equation}
\pi ^2_{\mathcal{B}^{1,1}}(t,\xi ,\zeta ,v)=\left (\begin{array}{ll}t^2\xi&t\xi\zeta \\t\xi\zeta&v\end{array}\right
).\label{EquationLocalCoordinateForB11Second}
\end{equation}
In this local coordinate system, $\mathcal{B}^{1,1}=\{t=0\}$. The set $\{t=0,\xi =\infty\}$ in the first projection $\pi ^1_{\mathcal{B}^{1,1}}$
corresponds to the set $\{t=0,\xi =0\}$ in this second projection $\pi ^2_{\mathcal{B}^{1,1}}$.

Let $K_{Z_4}=\pi _{Z_4}^{-1}\circ K\circ \pi _{Z_4}$ be the induced map of $K$ in $Z_4$.

\begin{proposition}
For $1\leq i\leq q$:

i) ${K}_{Z_4}(\mathcal{A}^{i,i})=\mathcal{B}^{i,i}\cap I(\Sigma _{i,i})$. In fact, if $(s=0,\zeta ,v)\in \mathcal{A}^{1,1}$ as in
(\ref{EquationLocalCoordinateForA11}) then
\begin{equation}
K_{Z_4}(s=0,\zeta ,v)=(t=0,\xi ',\zeta ',v' )\in \mathcal{B}^{1,1},\label{EquationImageOfA11}
\end{equation}
 where
$$\left (\begin{array}{ll}\xi '&\zeta '\\ \zeta '&v'\end{array}\right )=
I\left (\begin{array}{ll}0/\zeta _{1,1}&1/\zeta \\ 1/\zeta&1/v\end{array}\right ).$$

ii) ${K}_{Z_4}(\mathcal{B}^{i,i})=\mathcal{B}^{i,i}$.

Moreover, the restriction of ${K}_{Z_4}$ to each of the spaces $\mathcal{B}^{i,i}$ is the same as ${K}$, in the sense that

$${K}_{Z_4}(t=0,\xi ,\zeta ,v)=(t=0,\xi ',\zeta ', v'),$$
at generic points $(t=0,\xi ,\zeta ,v)$ of $\mathcal{B}^{1,1}$, where
$$\left (\begin{array}{ll}\xi '&\zeta '\\ \zeta '&v'\end{array}\right )={K}\left (\begin{array}{ll}\xi&\zeta \\ \zeta&v\end{array}\right ).$$
Similar results hold for the other $\mathcal{B}^{i,i}$'s ($1\leq i\leq q$). \label{PropsitionBasicPropertyOfKZ}\end{proposition}
\begin{proof}

i) We make use of the following property (see formula (4.4) in \cite{bedford-tuyen}):

If
$${K}\left (\begin{array}{ll}\xi&\zeta \\ \zeta&v\end{array}\right )=\left (\begin{array}{ll}\xi '&\zeta '\\ \zeta '&v'\end{array}\right )$$
then
\begin{equation}
{K}\left (\begin{array}{ll}t^2\xi&t\zeta \\t\zeta&v\end{array}\right )=\left (\begin{array}{ll}t^2\xi '&t\zeta '\\t\zeta '&v'\end{array}\right
).\label{EquationHomogeneousPropertyOfK}
\end{equation}

Using the projection (\ref{EquationLocalCoordinateForA11}), to determine $K_{Z_4}(\mathcal{A}^{1,1})$ it suffices to compute the limit when $s\rightarrow
0$ of $K(x)$ where
$$x=\left (\begin{array}{ll}s\zeta&s\zeta \\s\zeta&v\end{array}\right ).$$

Rewriting $x$ as
$$x=\left (\begin{array}{lll}s^2\zeta _{1,1}/s&s\zeta\\s\zeta&v\end{array}\right ),$$
using the formula (\ref{EquationHomogeneousPropertyOfK}), we have
$$K(x)=\left (\begin{array}{ll}s^2\xi '&s\zeta '\\s\zeta '&v'\end{array}\right ),$$
where
$$\left (\begin{array}{ll}\xi '&\zeta '\\ \zeta '&v'\end{array}\right )={K}\left (\begin{array}{ll}\zeta _{1,1}/s&\zeta \\ \zeta&v\end{array}\right )=
I\left (\begin{array}{ll}s/\zeta _{1,1}&1/\zeta \\ 1/\zeta&1/v\end{array}\right ).$$ The last formula shows that when $s\rightarrow 0$, the limit of
$K(x)$ is in $\mathcal{B}^{1,1}\cap I(\Sigma _{1,1})$, and we obtain (\ref{EquationImageOfA11}). Hence $K_{Z_4}(\mathcal{A}^{1,1})= \mathcal{B}^{1,1}\cap
I(\Sigma _{1,1})$.

The proof of ii) is similar.
\end{proof}

Let us consider a matrix
$$x=\left (\begin{array}{ll}\xi&\zeta \\ \zeta&v\end{array}\right ),$$
written as in (\ref{EquationLocalCoordinateForB11}). That is, $\xi$ and the $\zeta 's$ fill out the first row and column, where $\xi \in \mathbb{C}$. We
will consider algebraic subvarieties $W\subset \mathcal{S}_q$ with the property that whenever $x\in W$, then
\begin{equation}\left (\begin{array}{ll}t^2\xi&t\zeta \\
t\zeta&v\end{array}\right )\in W,\label{EquationCompatibility}\end{equation} for all $\mathbb{C}\ni t\not= 0$. If $W$ has this property, and if no
component of $W$ is contained in the indeterminacy loci of $I$, $J$, and $K$, then so do $I(W)$, $J(W)$, and $K(W)$.

We say that an irreducible hypersurface $W\subset\mathcal{S}_q$ is compatible with $\mathcal{B}^{1,1}$ if condition (\ref{EquationCompatibility}) is
satisfied and if moreover
$$W\not\subseteq
JR_{q-1}\cup \bigcup _{(k,l)\not= (1,1)}\Sigma _{k,l}.$$ When $W$ is compatible, then $W$ is not contained in any of the centers of blowups in the
construction of $Z_4$, thus we can take its strict transform inside $Z_4$ and define $\mathcal{B}^{1,1}\cap W\subset Z_4$. Using coordinate projections
analogous to (\ref{EquationLocalCoordinateForB11}), we may also define what it means for $W$ to be compatible with $\mathcal{B}^{i,i}$ for $2\leq i\leq
q$. Note that both hypersurfaces $\Sigma _{1,1}$ and $I(\Sigma _{1,1})$ are compatible with $\mathcal{B}^{1,1}$.

\begin{proposition}
For $1\leq i\leq q$:

If $W$ is compatible with $\mathcal{B}^{i,i}$ and $W\not\subseteq \Sigma _{i,i}$, then ${K}_{Z_4}(\mathcal{B}^{i,i}\cap W)=\mathcal{B}^{i,i}\cap {K}(W)$.

If $W=\Sigma _{i,i}$, then $K_{Z_4}(K_{Z_4}(\mathcal{B}^{i,i}\cap \Sigma _{i,i}))=\mathcal{B}^{i,i}\cap I(\Sigma _{i,i})$.

Moreover, $K_{Z_4}(\mathcal{B}^{i,i}\cap \Sigma _{i,i})$ can be written explicitly. For example, if $i=1$ then in the local coordinate system
(\ref{EquationLocalCoordinateForB11Second}) we have: $K_{Z_4}(\mathcal{B}^{1,1}\cap \Sigma _{1,1})=\{t=\xi =0\}$.

\label{PropsitionBasicPropertyOfKZ2}\end{proposition}
\begin{proof}
The first claim follows from the discussion in last paragraph and Proposition \ref{PropsitionBasicPropertyOfKZ}.

The proof of the third claim is similar to that of Proposition \ref{PropsitionBasicPropertyOfKZ3} iii).

The second claim follows from the third claim and an argument similar to that of the proof of Proposition \ref{PropsitionBasicPropertyOfKZ} i).
\end{proof}

E) Next we let $\pi _5:Z_5\rightarrow Z_4$ be the blow up of $Z_4$ along the strict transforms of $C_{i,j}=\mathcal{A}^{i,j}\cap \Sigma _{i,i}\cap \Sigma
_{j,j}$ (where $1\leq i<j\leq q$), with exceptional divisors $\mathcal{C}^{i,j}$. We describe a local coordinate system of $\pi _5$ near the exceptional
divisor $\mathcal{C}^{1,2}$. We fix $3\leq i_0,j_0\leq q$, $1\leq \min\{k_0,l_0\}\leq 2$, $k_0\not= l_0$. Let $t\in \mathbb{C}$; $v=(v_{i,j})_{3\leq
i,j\leq q}\in \mathcal{S}_{q-2}$ and $v_{i_0,j_0}=1$; $\xi =(\xi_{1,1},~\xi _{2,2})\in\mathbb{C}^2$;
  $\zeta =(\zeta _{k,l})_{1\leq \min\{k,l\}\leq 2,~k\not= l}$, and $\zeta
_{k_0,l_0}=1$. In the local coordinate $(t,\xi ,\zeta , v )$, the projection $\pi _5=\pi _{\mathcal{C}^{1,2}}$ is given by

\begin{equation}\pi _{\mathcal{C}^{1,2}}(t,\xi ,\zeta ,v)=\left (\begin{array}{lll}t^2\xi _{1,1}&t\zeta&t\zeta \\t\zeta&t^2\xi _{2,2}&t\zeta
\\t\zeta&t\zeta&v\end{array}\right ).\label{EquationLocalCoordinateForC12}
\end{equation}
In this local coordinate system, $\mathcal{C}^{1,2}=\{t=0\}$.

F) Finally,  we let $\pi _6:Z_6\rightarrow Z_5$ be the blow up of $Z_5$ along the strict transforms of $D_{i,j}=\mathcal{C}^{i,j}\cap \Sigma _{i,j}$
(where $1\leq i<j\leq q$), with exceptional divisors $\mathcal{D}^{i,j}=\pi _6^{-1}(D_{i,j})$. We describe two local coordinate systems of $\pi _6$ near
the exceptional divisor $\mathcal{D}^{1,2}$.

For the first local coordinate system, we fix $3\leq i_0,j_0\leq q$, $1\leq \min\{k_0,l_0\}\leq 2<\max\{k_0,l_0\}$. Let $t\in \mathbb{C}$;
$v=(v_{i,j})_{3\leq i,j\leq q}\in \mathcal{S}_{q-2}$ and $v_{i_0,j_0}=1$; $\xi =(\xi_{1,1},~\xi _{1,2},~\xi _{2,2})\in\mathbb{C}^3$;
 $\zeta =(\zeta _{k,l})_{1\leq \min\{k,l\}\leq 2<\max \{k,l\}}$, and $\zeta
_{k_0,l_0}=1$. In the local coordinate $(t,\xi ,\zeta , v )$, the projection $\pi _6=\pi ^1_{\mathcal{D}^{1,2}}$ is given by
\begin{equation}
\pi ^1_{\mathcal{D}^{1,2}}(t,\xi ,\zeta ,v)=\left (\begin{array}{lll}t^2\xi _{1,1}&t^2\xi _{1,2}&t\zeta \\t^2\xi _{1,2}&t^2\xi _{2,2}&t\zeta\\t\zeta
&t\zeta &v\end{array}\right ). \label{EquationLocalCoordinateForD12}
\end{equation}
In this local coordinate system, $\mathcal{D}^{1,2}=\{t=0\}$.

To cover the points corresponding to $\xi _{1,2}=\infty$ in the first projection $\pi ^1_{\mathcal{D}^{1,2}}$, we let $t\in \mathbb{C}$;
$v=(v_{i,j})_{3\leq i,j\leq q}\in \mathcal{S}_{q-2}$ and $v_{i_0,j_0}=1$; $\lambda \in\mathbb{C}$; $\xi =(\xi_{1,1},~\xi _{1,2},~\xi
_{2,2})\in\mathbb{C}^3$ and one of its coordinates is $1$;
 $\zeta =(\zeta _{k,l})_{1\leq \min\{k,l\}\leq 2<\max \{k,l\}}$, and $\zeta
_{k_0,l_0}=1$. In the local coordinate $(t,\xi ,\zeta , v )$, the projection $\pi _6=\pi ^2_{\mathcal{D}^{1,2}}$ is given by
\begin{equation}
\pi ^2_{\mathcal{D}^{1,2}}(t, \lambda ,\xi ,\zeta ,v)=\left (\begin{array}{lll}t^2\lambda ^2\xi_{1,1}&t^2\lambda \xi_{1,2}&t\lambda \zeta \\t^2\lambda
\xi _{1,2}&t^2\lambda ^2\xi _{2,2}&t\lambda \zeta\\t\lambda\zeta &t\lambda\zeta &v\end{array}\right ).\label{EquationLocalCoordinateForD12Second}
\end{equation}
In this local coordinate system, $\mathcal{D}^{1,2}=\{t=0\}$. The set $\{t=0, \xi _{1,2}=\infty\}$ in the first projection $\pi ^1_{\mathcal{D}^{1,2}}$
corresponds to the set $\{t=0, \lambda =0\}$ in this second projection $\pi ^2_{\mathcal{D}^{1,2}}$.

F) We define $Z=Z_6$. Let ${K}_Z=\pi _Z^{-1}\circ {K}\circ \pi _Z:Z\rightarrow Z$ be the induced map of ${K}$ on $Z$.

\begin{proposition}

For $1\leq i<j\leq q$:

i) $K_Z(\Sigma _{i,j})=\mathcal{C}^{i,j}$.

ii) ${K}_Z(\mathcal{A}^{i,j})=\mathcal{D}^{i,j}\cap I(\Sigma _{i,i}\cap \Sigma _{j,j}\cap \Sigma _{i,j})$.

iii) ${K}_Z(\mathcal{C}^{i,j})=\mathcal{D}^{i,j}\cap I(\Sigma _{i,j})$.

iv)  ${K}_Z(\mathcal{D}^{i,j})=\mathcal{D}^{i,j}$.

Moreover, the restriction of ${K}_Z$ to each of the spaces $\mathcal{D}^{i,j}$ is the same as ${K}$, in the sense that

$${K}_Z(t=0,\xi ,\zeta ,v)=(t=0,\xi ',\zeta ',v'),$$
at generic points $(t=0,\xi ,\zeta ,v)$ of $\mathcal{D}^{1,2}$, where
$$\left (\begin{array}{lll}\xi '&\xi '&\zeta '\\ \xi '&\xi '&\zeta '
\\ \zeta '&\zeta '&v '\end{array}\right )={K}\left (\begin{array}{lll}\xi& \xi&\zeta \\ \xi&\xi&\zeta\\ \zeta &\zeta &v\end{array}\right ).$$

Similar results hold for other $\mathcal{D}^{i,j}$'s ($1\leq i<j\leq q$).

\label{PropsitionBasicPropertyOfKZ1}\end{proposition}
\begin{proof}
The proofs of all these claims are similar to the proof of Proposition \ref{PropsitionBasicPropertyOfKZ},  but instead of using formula
(\ref{EquationHomogeneousPropertyOfK}), we use a similar formula:

If
$${K}\left (\begin{array}{lll}\xi& \xi&\zeta \\ \xi&\xi&\zeta\\ \zeta &\zeta &v\end{array}\right )=\left (\begin{array}{lll}\xi '&\xi '&\zeta '\\ \xi '&\xi '&\zeta '
\\ \zeta '&\zeta '&v '\end{array}\right )$$
then
$${K}\left (\begin{array}{lll}t^2\xi& t^2\xi&t\zeta \\ t^2\xi&t^2\xi& t\zeta\\ t\zeta &t\zeta &v\end{array}\right )=\left (\begin{array}{lll}t^2\xi '&t^2\xi '&t\zeta '
\\ t^2\xi '&t^2\xi '&t\zeta ' \\ t\zeta '&t\zeta '&v '\end{array}\right ).$$
\end{proof}

\begin{corollary} The exceptional hypersurfaces of $K_Z$
are $\mathcal{A}^{i,i}$ (for $1\leq i\leq q$), $\mathcal{A}^{i,j}$ (for $1\leq i<j\leq q$), and $\mathcal{C}^{i,j}$ (for $1\leq i<j\leq q$).
\label{Corollary0}\end{corollary}

Let us consider a matrix
$$x=\left (\begin{array}{lll}\xi _{1,1}&\xi _{1,2}&\zeta \\ \xi _{1,2}&\xi _{2,2}&\zeta\\ \zeta
&\zeta &v\end{array}\right ),$$ written as in (\ref{EquationLocalCoordinateForD12}). That is, the $\xi$'s and $\zeta 's$ fill out first two rows and
first two columns. We will consider algebraic subvarieties $W\subset \mathcal{S}_q$ with the property that whenever $x\in W$, then
\begin{equation}\left (\begin{array}{lll}t^2\xi _{1,1}&t^2\xi _{1,2}&t\zeta \\t^2\xi _{1,2}&t^2\xi _{2,2}&t\zeta\\t\zeta
&t\zeta &v\end{array}\right )\in W,\label{EquationCompatibility1}\end{equation} for all $\mathbb{C}\ni t\not= 0$. If $W$ has this property, and if no
component of $W$ is contained in the indeterminacy loci of $I$, $J$, and $K$, then so do $I(W)$, $J(W)$, and $K(W)$.

We say that an irreducible hypersurface $W$ is compatible with $\mathcal{D}^{1,2}$ if condition (\ref{EquationCompatibility1}) is satisfied and if
moreover
$$W\not\subseteq
JR_{q-1}\cup \bigcup _{(k,l)\not= (1,1),(1,2),(2,2)}\Sigma _{k,l}.$$ When $W$ is compatible, then $W$ is not contained in any of the centers of blowups
in the construction of $Z$, thus we can take its strict transform inside $Z$ and define $\mathcal{D}^{1,2}\cap W\subset Z$. Using coordinate projections
analogous to (\ref{EquationLocalCoordinateForD12}), we may also define what it means for $W$ to be compatible with $\mathcal{D}^{k,l}$ for $1\leq k<l\leq
q$. Note that both hypersurfaces $\Sigma _{1,2}$ and $I(\Sigma _{1,2})$ are compatible to $\mathcal{D}^{1,2}$.

Similarly to Proposition \ref{PropsitionBasicPropertyOfKZ2}, we obtain

\begin{proposition}
For $1\leq i<j\leq q$:

If $W$ is compatible with $\mathcal{D}^{i,j}$ and $W\not\subseteq \Sigma _{i,i}\cup \Sigma _{i,j}\cup \Sigma _{j,j}$, then ${K}_{Z}(\mathcal{D}^{i,j}\cap
W)=\mathcal{D}^{i,j}\cap {K}(W)$.

If $W=\Sigma _{i,j}$, then $K_{Z}(K_{Z}(\mathcal{D}^{i,j}\cap \Sigma _{i,j}))=\mathcal{D}^{i,j}\cap I(\Sigma _{i,j})$.

Moreover, $K_{Z}(\mathcal{D}^{i,j}\cap \Sigma _{i,j})$ can be explicitly written. For example, if $i=1,j=2$, then in the local coordinate system
(\ref{EquationLocalCoordinateForD12Second}) we have:  $K_{Z}(\mathcal{D}^{1,2}\cap \Sigma _{1,2})=\{t=\lambda =0\}$.
\label{PropsitionBasicPropertyOfKZ22}\end{proposition}

\section{A lower bound for $\delta ({K})$}

We will use the notation: $$S=\bigcup _{i\not= j}\mathcal{A}^{i,j}, ~U=Z\backslash S.$$ In this section we will show that instead of establishing the
property (\ref{Equation11Regular}) for $K_Z$, we can work with the restriction of $K_Z$ to the Zariski dense open subset $U$ of $Z$.

We denote by $\mathcal{I}(K_Z)$ the indeterminacy locus of $K_Z$.

\begin{lemma}
For any $n\geq 1$, and for any $1\leq i<j\leq q$:

${K}_Z^n(\mathcal{A}^{i,i})$ is a subvariety of codimension $1$ of $\mathcal{B}^{i,i}$, and is not contained in $\mathcal{I}(K_Z)\cup S$.

${K}_Z^n(\mathcal{C}^{i,j})$ is a subvariety of codimension $1$ of $\mathcal{D}^{i,j}$, and is not contained in $\mathcal{I}(K_Z)\cup S$.
\label{LemmaOrbitsOfExceptionalHypersurfacesOfKZ}\end{lemma}
\begin{proof}
In the following, as noted before, we assume that $q\geq 5$. We present the proof only for $\mathcal{A}^{1,1}$, since the proofs for other
$\mathcal{A}^{i,i}$'s and for $\mathcal{C}^{i,j}$'s are similar.

By Proposition \ref{PropsitionBasicPropertyOfKZ}, we know that ${K}_Z(\mathcal{A}^{1,1})=\mathcal{B}^{1,1}\cap I(\Sigma _{1,1})$. Hence from Proposition
\ref{PropsitionBasicPropertyOfKZ2}, as long as ${K}^m(I(\Sigma _{1,1}))\not\subset JR_{q-1}\cup\bigcup _{k,l}\Sigma _{k,l}$ for all $m=0,\ldots ,n$ then
${K}_{Z}^{m+1}(\mathcal{A}^{1,1})=\mathcal{B}^{1,1}\cap {K}^m(I(\Sigma _{1,1}))$, for all $m=0,\ldots ,n$. Each of these varieties is a subvariety of
codimension $1$ of $\mathcal{B}^{1,1}$, and is not contained in the indeterminacy locus of ${K}_Z$. Moreover, ${K}^m(I(\Sigma _{1,1}))$ is then
compatible to $\mathcal{B}^{1,1}$, hence $\mathcal{B}^{1,1}\cap {K}^m(I(\Sigma _{1,1}))$ is defined in the local coordinate
(\ref{EquationLocalCoordinateForB11}) by $\{t=0, P(\xi ,\zeta ,v) =0)\}$ where $P(x_{i,j})=0$ is the equation in $\mathcal{S}_q$ of ${K}^m(I(\Sigma
_{1,1}))$. From this, it is easy to see that $\mathcal{B}^{1,1}\cap {K}^m(I(\Sigma _{1,1}))$ is not contained in $\bigcup _{k\not= l}\mathcal{A}^{k,l}$.

Hence it remains to explore what happens in case ${K}^n(I(\Sigma _{1,1}))\subset JR_{q-1}\cup\bigcup _{k,l}\Sigma _{k,l}$ for some $n$. We choose $n=n_0$
to be the smallest integer satisfying ${K}^n(I(\Sigma _{1,1}))\subset JR_{q-1}\cup\bigcup _{k,l}\Sigma _{k,l}$. It is not difficult to see that $I(\Sigma
_{1,1})\not\subset JR_{q-1}\cup\bigcup _{k,l}\Sigma _{k,l}$, hence $n_0>0$, and then by definition of $n_0$:
\begin{equation}
{K}^m(I(\Sigma _{1,1}))\not\subset JR_{q-1}\cup\bigcup _{k,l}\Sigma _{k,l}, \label{EquationLemmaOrbitsOfExceptionalHypersurfacesOfKZ1}\end{equation} for
all $m=0,\ldots ,n_0-1$, and
\begin{equation}
{K}^m(I(\Sigma _{1,1}))\subset JR_{q-1}\cup\bigcup _{k,l}\Sigma _{k,l}. \label{EquationLemmaOrbitsOfExceptionalHypersurfacesOfKZ2}\end{equation}

Since $I(\Sigma _{1,1})$ is an irreducible hypersurface, ${K}$ is a birational map, and since $JR_{q-1}$ and $\Sigma _{k,l}$'s are the only exceptional
hypersurfaces of ${K}$, (\ref{EquationLemmaOrbitsOfExceptionalHypersurfacesOfKZ1}) and (\ref{EquationLemmaOrbitsOfExceptionalHypersurfacesOfKZ2}) imply
that for all $m=0,\ldots ,n_0$: ${K}^m(I(\Sigma _{1,1}))$ is an irreducible hypersurface in $\mathcal{S}_q$. Moreover, either
\begin{equation}
{K}^{n_0}(I(\Sigma _{1,1}))= JR_{q-1}, \label{EquationLemmaOrbitsOfExceptionalHypersurfacesOfKZ3}\end{equation} or
\begin{equation}
{K}^{n_0}(I(\Sigma _{1,1}))=\Sigma _{i,j}, \label{EquationLemmaOrbitsOfExceptionalHypersurfacesOfKZ4}\end{equation} for some $1\leq i,j\leq q$.

Now we show that in fact
\begin{equation}
{K}^{n_0}(I(\Sigma _{1,1}))=\Sigma _{1,1}. \label{EquationLemmaOrbitsOfExceptionalHypersurfacesOfKZ5}\end{equation}

To this end, we will use the operations $\rho _{l,m}$ defined as follows: For $1\leq l,m\leq q$, let $\rho _{l,m}:\mathcal{S}_q\rightarrow \mathcal{S}_q$
denote the matrix operation which interchanges the $l$-th and $m$-th rows, and then interchanges the $l$-th and $m$-th columns of a matrix $x\in
\mathcal{S}_q$. Observe that on the space $\mathcal{S}_q$ : $\rho _{l,m}(I(x))=I(\rho _{l,m}(x))$, $\rho _{l,m}(J(x))=J(\rho _{l,m}(x))$, and $\rho
_{l,m}({K}(x))={K}(\rho _{l,m}(x))$. In particular, $\rho _{l,m}JR_{q-1}=JR_{q-1}$.

First we rule out the possibility (\ref{EquationLemmaOrbitsOfExceptionalHypersurfacesOfKZ3}). Assume in order to reach a contradiction that
${K}^{n_0}(I(\Sigma _{1,1}))= JR_{q-1}$. Then for all $i$ we have
$${K}^{n_0}(I(\Sigma _{i,i}))= {K}^{n_0}(I(\rho _{i,1}\Sigma _{1,1}))= \rho _{i,1}{K}^{n_0}(I(\Sigma _{1,1}))= \rho _{i,1}JR_{q-1}=JR_{q-1}.$$
Hence $q$ different irreducible hypersurfaces $I(\Sigma _{1,1}),\ldots ,I(\Sigma _{q,q})$ are mapped under ${K}^{n_0}$ to the same irreducible
hypersurfaces $JR_{q-1}$. But this would be a contradiction to the fact that ${K}^{n_0}$ is birational. Thus we showed that
(\ref{EquationLemmaOrbitsOfExceptionalHypersurfacesOfKZ3}) does not occur. Hence (\ref{EquationLemmaOrbitsOfExceptionalHypersurfacesOfKZ4}) must occur.

We next show that ${K}^{n_0}(I(\Sigma _{1,1}))=\Sigma _{1,1}$. We know that ${K}^{n_0}(I(\Sigma _{1,1}))=\Sigma _{i,j}$, for some $1\leq i,j\leq q$. We
need to show that $i=j=1$. Assume in order to reach a contradiction that $i\not= 1$ or $j\not= 1$. We have two cases:

Case 1: Both $i,j\not =1$. Choose $k\not= i,j,1$, we have then:
$${K}^{n_0}(I(\Sigma _{k,k}))={K}^{n_0}(I(\rho _{k,1}\Sigma _{1,1}))=\rho _{k,1}{K}^{n_0}(I(\Sigma _{1,1}))=\rho _{k,1}\Sigma _{i,j}=\Sigma _{i,j}.$$
Hence two different irreducible hypersurfaces $I(\Sigma _{1,1})$ and $I(\Sigma _{k,k})$ have the same image $\Sigma _{i,j}$ under the birational mapping
${K}^{n_0}$, which is a contradiction.

Case 2: One of $i,j$ is $1$, but the other is not. Without loss of generality, we may assume that $i=1$ and $j\not =1$. Then
$${K}^{n_0}(I(\Sigma _{j,j}))={K}^{n_0}(I(\rho _{1,j}\Sigma _{1,1}))=\rho _{1,j}{K}^{n_0}(I(\Sigma _{1,1}))=\rho _{1,j}\Sigma _{1,j}=\Sigma _{1,j}.$$
Hence two different irreducible hypersurfaces $I(\Sigma _{1,1})$ and $I(\Sigma _{j,j})$ have the same image $\Sigma _{1,j}$ under the birational map
${K}^{n_0}$, which is again a contradiction.

Hence we showed that if $n_0>0$ is the smallest integer such that ${K}^{n_0}(I(\Sigma _{1,1}))\subset JR_{q-1}\cup\bigcup _{k,l}\Sigma _{k,l}$, then for
all $m=0,\ldots ,n_0$, ${K}^m(I(\Sigma _{1,1}))$ is an irreducible hypersurface of $\mathcal{S}_q$, and ${K}^{n_0}(I(\Sigma _{1,1}))=\Sigma _{1,1}$.
Hence by Proposition \ref{PropsitionBasicPropertyOfKZ2}, for all $m=0,\ldots ,n_0$: ${K}_Z^{m}(\mathcal{B}^{1,1}\cap I(\Sigma
_{1,1}))=\mathcal{B}^{1,1}\cap {K}^m(I(\Sigma _{1,1}))$ is a subvariety of codimension $1$ of $\mathcal{B}^{1,1}$, and such that (by Proposition
\ref{PropsitionBasicPropertyOfKZ}) ${K}_Z^{n_0+1}(\mathcal{B}^{1,1}\cap I(\Sigma _{1,1}))={K}_Z(\mathcal{B}^{1,1}\cap \Sigma _{1,1})$ is a subvariety of
codimension $1$ of $\mathcal{B}^{1,1}$. Moreover
$${K}_Z^{n_0+2}(\mathcal{B}^{1,1}\cap I(\Sigma _{1,1}))={K}_Z({K}_Z(\mathcal{B}^{1,1}\cap \Sigma _{1,1}))=\mathcal{B}^{1,1}\cap I(\Sigma _{1,1})={K}_Z(\mathcal{A}^{1,1}).$$
Hence if (\ref{EquationLemmaOrbitsOfExceptionalHypersurfacesOfKZ2}) happens, then the orbit of $K_Z(\mathcal{A}^{1,1})$ under ${K}_Z$ is periodic. Thus
the orbit of $K_Z(\mathcal{A}^{1,1})$ never lands in $\mathcal{I}(K_Z)$.

To complete the proof, we need to show that the orbit never lands in $S=\bigcup _{i\not= j}\mathcal{A}^{i,j}$. That $K_Z^{n_0}(\mathcal{B}^{1,1}\cap
I(\Sigma _{1,1}))$, which equals $\mathcal{B}^{1,1}\cap \Sigma _{1,1}$, is not contained in $S$ can be checked directly. For values $m$ when
$K_Z^{m}(\mathcal{B}^{1,1}\cap I(\Sigma _{1,1}))\not= \mathcal{B}^{1,1}\cap \Sigma _{1,1}$, we can use the argument at the end of the second paragraph of
this proof to show that $K_Z^{m}(\mathcal{B}^{1,1}\cap I(\Sigma _{1,1}))$ (which is then equal to $\mathcal{B}^{1,1}\cap K_Z^m(I(\Sigma _{1,1}))$) is not
contained in $S$ as well.
\end{proof}

By Lemma \ref{LemmaOrbitsOfExceptionalHypersurfacesOfKZ}, we obtain the following result
\begin{corollary}
If $V$ is an irreducible hypersurface which is not contained in $S$ then for any $n\geq 1$: ${K}_Z^n(V)$ is not contained in $\mathcal{I}({K}_Z)\cup
S$.\label{Corollary1}\end{corollary}

Let $V$ be a hypersurface (or divisor) of $Z$. We let $V|_{U}$ denote the restriction to $U$. Let $R_U(V)$ denote the "extension by zero" of $V|_{U}$ to
$Z$. We let $(K_Z^n)^*(V)$ denote the pull-back of $V$ by the map $K_Z^n$.
\begin{proposition}
If $V$ is a hypersurface on $Z$, then for all $n\geq 1$:
\begin{equation}
R_U(({K}_Z^n)^*V)=R_U(({K}_Z^n)^*R_U(V))=R_U(({K}_Z^*)^nV)=R_U(({K}_Z^*)^nR_U(V)), \label{Theorem111RegularOnOpenSet}\end{equation} as divisors on $Z$.
In particular, if $R_U(V)=0$ then for all $n\geq 1$: $R_U(({K}_Z^n)^*V)=0$. \label{Theorem11RegularOnOpenSet}\end{proposition}
\begin{proof}
Before applying $R_U$ on the left, the difference between any two of the divisors in equation (\ref{Theorem111RegularOnOpenSet}) is a hypersurface
supported in $K_Z^{-j}(\mathcal{I}(K_Z)\cup S)$. However, by Corollary \ref{Corollary1}, this last set is disjoint from $U$, hence the difference
vanishes on applying $R_U$.
\end{proof}

Define $\Lambda :=Pic(Z)/ker (R_U)$, and let $pr_{\Lambda}:Pic(Z)\rightarrow \Lambda$ be the canonical projection. By Proposition
\ref{Theorem11RegularOnOpenSet}, the maps $pr_{\Lambda}\circ ({K}_Z^n)^*:Pic(Z)\rightarrow \Lambda$ induce well-defined maps $L_n:\Lambda\rightarrow
\Lambda$ which satisfy the identities: $L_n=(L_1)^n$ for all $n\geq 1$.

\begin{theorem}
$\delta ({K})\geq sp(L_1)$, where $sp(L_1)$ is the spectral radius of $L_1$. \label{TheoremLowerBoundForLambda1}\end{theorem}
\begin{proof}
The dynamical degree $\delta (K_Z)=\lim _{n\rightarrow \infty}||(K_Z^n)^*||^{1/n}$ is independent of the choice of norm $||.||_{Pic(Z)}$ on $Pic(Z)$.
Further, since $\pi _Z $ is a birational map, we have that $\delta (K_Z)=\delta (K)$ (see for example \cite{dinh-sibony}, and see \cite{dinh-nguyen} for
more general results). Finally, if we use the induced norm on $\Lambda$, we have
\begin{eqnarray*}
\lim _{n\rightarrow\infty}||({K}_Z^n)^*||_{Pic(Z)}^{1/n} \geq\lim _{n\rightarrow\infty}||L_n||_{\Lambda}^{1/n}=\lim
_{n\rightarrow\infty}||(L_1)^n||_{\Lambda}^{1/n}=sp(L_1).
\end{eqnarray*}

\end{proof}

\section{The spectral radius of $L_1$}
A basis for the Picard group $Pic(Z)$ is given by $H$ (the class of a generic hyperplane in $\mathcal{S}_q$), and the classes of the strict transforms of
$\mathcal{R}^1$, $\mathcal{A}^{i,i}$'s ($1\leq i\leq q$), $\mathcal{B}^{i,i}$'s ($1\leq i\leq q$), $\mathcal{A}^{i,j}$'s ($1\leq i<j\leq q$),
$\mathcal{C}^{i,j}$'s ($1\leq i<j\leq q$), and $\mathcal{D}^{i,j}$'s ($1\leq i<j\leq q$). The images under $pr_{\Lambda}$ of classes of $H$ and of the
strict transforms of $\mathcal{R}^1$, $\mathcal{A}^{i,i}$ ($1\leq i\leq q$), $\mathcal{B}^{i,i}$ ($1\leq i\leq q$), $\mathcal{C}^{i,j}$ ($1\leq i<j\leq
q$), and $\mathcal{D}^{i,j}$ ($1\leq i<j\leq q$) form a basis for $\Lambda$. For convenience, we will use the same letters to denote the images of these
classes in $\Lambda$. Further, we define
\begin{equation}
\mathcal{A}=\sum _{i}\mathcal{A}^{i,i},~\mathcal{B}=\sum _{i}\mathcal{B}^{i,i},~\mathcal{C}=2\sum _{i<j}\mathcal{C}^{i,j},~ \mathcal{D}=2\sum
_{i<j}\mathcal{D}^{i,j}.\label{EquationBasisOfLambdaZero}
\end{equation}

Let $\Lambda _0$ be the subspace of $\Lambda$ generated by the ordered basis $H,~\mathcal{R}^1,~\mathcal{A},~\mathcal{B},~\mathcal{C}$ and $\mathcal{D}$.

\begin{lemma}
The map $L_1$ restricted to $\Lambda _0$ is given by

\begin{eqnarray*}
L_1(H)&=&(q^2-q+1)H-(q-2)\mathcal{R}^1-(2q-3)\mathcal{A}-(2q-2)\mathcal{B}-(2q-3)\mathcal{C}-(2q-2)\mathcal{D},\\
L_1(\mathcal{R}^1)&=&(q^2-q)H-(q-1)\mathcal{R}^1-(2q-3)\mathcal{A}-(2q-2)\mathcal{B}-(2q-3)\mathcal{C}-(2q-2)\mathcal{D},\\
L_1(\mathcal{A})&=&qH-\mathcal{A}-2\mathcal{B}-2\mathcal{C}-2\mathcal{D},\\
L_1(\mathcal{B})&=&\mathcal{A}+\mathcal{B},\\
L_1(\mathcal{C})&=&(q^2-q)H-(2q-2)\mathcal{A}-(2q-2)\mathcal{B}-(2q-3)\mathcal{C}-(2q-2)\mathcal{D},\\
L_1(\mathcal{D})&=&\mathcal{C}+\mathcal{D}.
\end{eqnarray*}
In particular, $\Lambda _0$ is invariant under $L_1$, and the spectral radius of $L_1|\Lambda _0$ is the largest root of the polynomial $\lambda
^2-(q^2-4q+2)\lambda +1$. \label{LemmaTheMapL}\end{lemma}
\begin{proof}
The proof is similar to the proof of Proposition 6.1 in \cite{bedford-tuyen}. For example, we determine $L_1(H)$. There are integers $a,~b,~\alpha
_{i,i},~\beta _{i,i},~\gamma _{i,j}$ and $\lambda _{i,j}$ such that
\begin{eqnarray*}
L_1(H)&=&a H-b\mathcal{R}^1-\sum _{1\leq i\leq q}\alpha _{i,i}\mathcal{A}^{i,i}\\&&-\sum _{1\leq i\leq q}\beta _{i,i}\mathcal{B}^{i,i}-\sum _{1\leq
i<j\leq q}\gamma _{i,j}\mathcal{C}^{i,j}-\sum _{1\leq i<j\leq q}\lambda _{i,j}\mathcal{D}^{i,j}.
\end{eqnarray*}

By symmetry, there are constants $\alpha ,~\beta ,~\gamma$ and $~\lambda$ such that $\alpha _{i,i}=\alpha ,~\beta _{i,i}=\beta ,~\gamma _{i,j}=\gamma $
and $\lambda _{i,j}=\lambda $ for all $1\leq i<j\leq q$. Thus
\begin{eqnarray*}
L_1(H)=aH-b\mathcal{R}^1-\alpha \mathcal{A}-\beta \mathcal{B}-\frac{1}{2}\gamma \mathcal{C}-\frac{1}{2}\lambda \mathcal{D}.
\end{eqnarray*}

Recall from Proposition \ref{PropositionBasicPropertiesOfK} that the homogeneous form of $K$ is
\begin{eqnarray*}
\widehat{K}_{i,j}(x)=C_{i,j}(1/{x})\prod ({x}),
\end{eqnarray*}
where $x=(x_{k,l})_{1\leq k,l\leq q}\in \mathcal{S}_q$.

The coefficient $a$ is the degree of $K$, so by Proposition \ref{PropositionBasicPropertiesOfK}, we have $a=q^2-q+1$. To find the other coefficients, we
let $H=\{l=0\}$ where $l=\sum c_{i,j}x_{i,j}$, and we determine the order of vanishing of $\widehat{K}\circ l$ at the various divisors.

The constant $b$ is the order of vanishing of $\widehat{K}\pi _{\mathcal{R}^1}(s,v ,\nu )$ in $s$, where $\pi _{\mathcal{R}^1}$ is given in
(\ref{EquationLocalCoordinateForR1}). For $\nu =(\nu _1,\ldots ,\nu _q)$ with $\nu _1\ldots \nu _q\not= 0$, $\prod (\pi _{\mathcal{R}^1}(s,v ,\nu ))\not=
0$ when $s=0$. Further
$$\frac{1}{\pi _{\mathcal{R}^1}(s,v ,\nu )}=\frac{1}{\nu}\otimes \frac{1}{\nu}+O(s).$$
Since $\frac{1}{\nu}\otimes \frac{1}{\nu}$ has rank $1$, $C_{i,j}(1/\pi _{\mathcal{R}^1}(s,v ,\nu ))=O(s^{q-2})$. Thus $b=q-2$.

The constant $\alpha$ is the order of vanishing of $\widehat{K}\pi _{\mathcal{A}^{1,1}}(s,\zeta  ,v )$ in $s$, where $\pi _{\mathcal{A}^{1,1}}$ is given
in (\ref{EquationLocalCoordinateForA11}). The order of vanishing of $\prod (\pi _{\mathcal{A}^{1,1}}(s,\zeta ,v ))$ in $s$ is $2q-1$, since only the
entries on the first row and first column of the matrix $\pi _{\mathcal{A}^{1,1}}(s,\zeta ,v )$ vanish when $s=0$, and moreover all of these entries
vanishes to order $1$ in $s$. The minimal order of vanishing of $C_{i,j}(1/(\pi _{\mathcal{A}^{1,1}}(s,\zeta ,v )))$ ($1\leq i,j\leq q$) in $s$ is $-2$,
since $C_{i,j}(1/(\pi _{\mathcal{A}^{1,1}}(s,\zeta ,v )))$ is a sum whose summands are of the form $\pm \sigma  _1\sigma  _2\ldots \sigma _{q-1}$, where
$\sigma _{i}$ are entries of $1/\pi _{\mathcal{A}^{1,1}}(s,\zeta ,v ))$ and not any two of them are from a same row or column. Thus $\alpha =2q-3$.

The constants $\beta =2q-2$, $\gamma =4q-6$, and $\lambda =4q-4$ are similarly determined. Hence $L_1(H)$ is as in the statement of the lemma.
\end{proof}

\textit{Proof of Theorem \ref{theorem1}}: By Theorem \ref{TheoremLowerBoundForLambda1} and Lemma \ref{LemmaTheMapL}, we have $\delta ({K})\geq
sp(L_1)\geq sp(L_1|\Lambda _0)=$ the largest root of the polynomial $\lambda ^2-(q^2-4q+2)\lambda +1$. Because the degree complexity of the matrix
inversion restricted to $\mathcal{S}_q$ is not larger than that of the general matrices, and since the value of the later is equal to the largest root of
the polynomial $\lambda ^2-(q^2-4q+2)\lambda +1$ (see \cite{bedford-tuyen}), we conclude that $\delta ({K})=$ the largest root of the polynomial $\lambda
^2-(q^2-4q+2)\lambda +1$.



\begin{thebibliography}{xx}
\bibitem{auriac-maillard-viallet1}{J. C. Angles d'Auriac, J. M. Maillard, and C. M. Viallet,} \textit{A classification of four-state spin edge Potts models,}
J. Phys. A 35 (2002), 9251--9272.

\bibitem{auriac-maillard-viallet2}{J. C. Angles d'Auriac, J. M. Maillard, and C. M. Viallet,} \textit{On the complexity of some birational
transformations,} J. Phys. A: Math. Gen. 39 (2006), 3641--3654.

\bibitem{bedford-kim1}{E. Bedford and K-H Kim,}\textit{On the degree growth of birational mappings in higher dimension,} J. Geom. Anal. 14 (2004),
567--596.

\bibitem{bedford-kim2}{E. Bedford and K-H Kim,}\textit{Degree growth of matrix inversion: birational maps of symmetric, cyclic matrices,}
Discrete Con. Dyn. Syst. 21 (2008), no. 4, 977--1013.

\bibitem{bedford-tuyen}{Eric Bedford and Tuyen Trung Truong,} \textit{Degree complexity of birational maps related to matrix inversion,} Comm. Math. Phys 298 (2010),
no. 2, 357--368.

\bibitem{bellon-maillard-viallet}{M. P. Bellon,  J. M. Maillard, and C-M Viallet,} \textit{Integrable Coxeter groups,} Phys. Lett. A 159 (1991), 221--232.

\bibitem{bellon-viallet}{M. Bellon and C. M. Viallet,} \textit{Algebraic entropy,} Comm. Math. Phys. 204 (1999), 425--437.

\bibitem{boukraa-maillard}{S. Boukraa and J. M. Maillard, } \textit{Factorization properties of birational mappings,} Physica A 220 (1995), 403--470.

\bibitem{dinh-nguyen}{Tien-Cuong Dinh and Viet-Anh Nguyen,} \textit{Comparison of dynamical degrees for semi-conjugate meromorphic maps,} to appear in
Commentarii Math. Helv. . arXiv: 0903.2621.

\bibitem{dinh-sibony}{T-C Dinh and N. Sibony,} \textit{Une borne superieure pour l'entropie topologique d'une application rationnelle,} Ann. of Math.
(2) 161 (2005), no. 3, 1637--1644.

\bibitem{fornaess-sibony}{John Erik Fornaess and Nessim Sibony,} \textit{Complex dynamics in higher dimensions,} Several Complex Variables, Math. Sci.
Res. Inst. Publ. 37, 273--296.

\bibitem{preissmann-auriac-maillard}{E. Preissmann, J. C. Angles d'Auriac, and J. M. Maillard,}
\textit{Birational mappings and matrix sub-algebra from the Chiral-Potts model,} J. Math. Phys. 50 (2009), no. 1, 013302, 26 pages.

\end{thebibliography}
\end{document}